\documentclass{article}
\usepackage{amsmath}
\usepackage{amsfonts}
\usepackage{mathtools}
\usepackage{amssymb}
\usepackage{graphicx} 
\usepackage{xcolor}
\usepackage{amsthm}
\usepackage{enumitem}
\usepackage{CJKutf8}
\usepackage{algorithm2e}
\usepackage{comment}

\theoremstyle{plain}
\newtheorem{theorem}{Theorem}[section]

\newtheorem{lemma}[theorem]{Lemma}
\newtheorem{proposition}[theorem]{Proposition}

\theoremstyle{definition}

\newtheorem{conjecture}[theorem]{Conjecture}
\theoremstyle{remark}

\newcommand{\cyan}[1]{{\color{blue}#1}}
\newcommand{\red}[1]{{\color{red}#1}}

\title{
A Homogeneous Nullstellensatz for Joint Invariant Subspaces }
\author{Sizhuo Yan, Jianting Yang and Lihong Zhi}
\date{February  2026}

\begin{document}

\maketitle
\begin{abstract}
Jurij Vol\v{c}i\v{c} conjectured that a noncommutative polynomial $g$ belongs to the unital $\mathbb{K}$-algebra generated by finitely many noncommutative polynomials if and only if, for matrices of every size, every joint invariant subspace of the evaluations of the generators is also invariant under the evaluation of $g$. In this paper, we establish a homogeneous Nullstellensatz for joint invariant subspaces by proving that this equivalence holds whenever the generators are homogeneous. In contrast, we demonstrate that the statement fails in the general case, thereby settling the conjecture completely.
\end{abstract}

\section{Introduction}

Hilbert’s Nullstellensatz is a foundational result in commutative algebraic geometry, making it natural to seek extensions in the context of noncommutative polynomials. A key distinction between the commutative and noncommutative cases is that zeros of noncommutative polynomials admit several inequivalent definitions. This paper focuses on Jurij  Vol\v{c}i\v{c}'s conjecture on a Nullstellensatz for noncommutative polynomials formulated in terms of joint invariant subspaces.

Throughout this paper, $\mathbb{N}$ denotes the set of natural numbers and $\mathbb{K}$ denotes an algebraically closed field of characteristic zero. 
We use $\underline{x} = (x_1, x_2, \ldots, x_d)$ to represent a tuple of noncommutative variables,  where $d$ is a positive integer. 
Let $\mathbb{K}\langle \underline{x}\rangle$ denote the free $\mathbb{K}$-algebra generated by $x_1, \ldots, x_d$; equivalently, $\mathbb{K}\langle \underline{x}\rangle$ is the algebra of noncommutative polynomials in $d$ variables over $\mathbb{K}$. The noncommutative polynomial in $\mathbb{K}\langle \underline{x}\rangle$ has the following form
\[f=\sum_{i=1}^{N}\alpha_{i} \omega_i,~ \text{where} ~\omega_i\in \langle \underline{x} \rangle,~\alpha_i\in \mathbb{K},~N\in \mathbb{N}.\]
The notation $\langle \underline{x} \rangle$ is the set of words generated by $x_1, x_2, \ldots, x_d$. The degree of a word $\omega$ is its length, which is denoted by  $\mathrm{deg}(\omega)$.  The degree of a noncommutative polynomial $f$ is the  maximum length of a word appearing in $f$, we  denote it by $\deg(f)$. 
Let $\underline{X} = (X_1, X_2,\ldots, X_d) \in \mathrm{M_n}(\mathbb{K})^d$ denote a $d$-tuple of $n \times n$ matrices. The evaluation of a noncommutative polynomial $f \in \mathbb{K}\langle \underline{x} \rangle$ at $\underline{X}$ is denoted by $f(\underline{X})$. 

In \cite{10.1007/978-3-031-75326-8_21}, when noncommutative polynomials are evaluated at matrices, the author summarized different types of zeros, namely true zeros, determinantal zeros, directional zeros,  weak zeros and trical zeros. The author then also introduced the Nullstellensatze corresponding to zeros defined in different ways \cite{amitsur1957generalization,MR4366019,helton2007strong,helton2004positivstellensatz,klep2014tracial,klep2017null,salomon2018algebras}. Moreover, free *-algebra with involution and its *-representations also define analogous zeros and admit the corresponding real Nullstellensatze \cite{CHMSN2013,MR4366019,KSV2018,MR432612,Schmudgen2009}. 
It is worth recalling two Nullstellensatze that are closely related to the results of this paper. 
\paragraph{True zero set}
The true zero set of a  noncommutative polynomial $g$ is defined as the collection of all tuples $\underline{X} = (X_1, \ldots, X_d)\in \mathrm{M_n}(\mathbb{K})^d$ for some $n$,  such that $g(\underline{X}) = 0$. Similarly, the common true zero set of noncommutative polynomials $f_1, \ldots, f_l$ consists of tuples $X $ that satisfy $f_j(\underline{X}) = 0$ for all $j = 1, \ldots, l$. Formally, these sets are denoted by the following:
\[Z(g):=\bigcup_{n \in \mathbb{N}}\left\{
X \in \mathrm{M_n}(\mathbb{K})^d\mid
g(\underline{X}) = 0 \right\}.\]
\[Z(f_1, \ldots, f_l):=\bigcup_{n \in \mathbb{N}}\left\{X  \in \mathrm{M_n}(\mathbb{K})^d  \mid f_j(\underline{X})= 0 \text{ for all } j = 1, \ldots, l\right\}.\]
The Nullstellensatz corresponding to true zeros under the
additional assumption that $f_1, \ldots, f_l$ are homogeneous is as follows:
\begin{theorem}[Salomon-Shalit-Shamovich]\label{Truezero}
\cite{salomon2018algebras} For $f_1,\ldots,f_l,\, g \in \mathbb{K}\langle \underline{x}\rangle$, suppose $f_1,\ldots,f_l$ be homogeneous, the following are equivalent:
\begin{enumerate}
\item[(i)] 
$Z(f_1, \ldots, f_l)  \subseteq Z(g)$;
\item[(ii)] the polynomial $g$ belongs to the two-sided ideal generated by polynomials $f_1,\ldots,f_l$.
\end{enumerate}    
\end{theorem}

\paragraph{Directional zero set}
The set of directional zeros of a  noncommutative polynomial $g$ is defined as the collection of all pairs $(\underline{X}, v)$, where $\underline{X} = (X_1, \ldots, X_d)$ is a $d$-tuple of $n \times n$ matrices over $\mathbb{K}$ for some $n$, and $v \in \mathbb{K}^n$ is a vector such that $g(\underline{X})v = 0$. Similarly, the common directional zero set of noncommutative polynomials $f_1, \ldots, f_l$ consists of a pair $(\underline{X}, v)$ that satisfies $f_j(\underline{X})v = 0$ for all $j = 1, \ldots, l$.  These sets are denoted by the following:
\[Z^{\mathrm{dir}}(g):=\bigcup_{n \in \mathbb{N}}\left\{
(\underline{X}, v) \in \mathrm{M_n}(\mathbb{K})^d \times \mathbb{K}^n\mid
g(\underline{X})\, v = 0 \right\}.\]
\[Z^{\mathrm{dir}}(f_1, \ldots, f_l):=\bigcup_{n \in \mathbb{N}}\left\{(\underline{X}, v) \in \mathrm{M_n}(\mathbb{K})^d \times \mathbb{K}^n\mid f_j(\underline{X})\, v = 0 \text{ for all } j = 1, \ldots, l\right\}.\]

The Nullstellensatz corresponding to directional zeros is as follows:
\begin{theorem}[Bergman]
\cite{helton2007strong,helton2004positivstellensatz}
For $f_1,\ldots,f_l,\, g \in \mathbb{K}\langle \underline{x}\rangle$, the following are equivalent:
\begin{enumerate}
\item[(i)]$Z^{\mathrm{dir}}(f_1, \ldots, f_l)  \subseteq Z^{\mathrm{dir}}(g)$;
\item[(ii)] the polynomial $g$ belongs to the left ideal generated by polynomials $f_1,\ldots,f_l$.
\end{enumerate}
\end{theorem}
In \cite{10.1007/978-3-031-75326-8_21}, the author explained that condition (i) in the above theorem can be interpreted as the condition that all  joint kernel spaces of matricial values of $f_1,\ldots,f_l$ are the kernel spaces of matricial values of the polynomial $g$, and then proposed the following conjecture by extending kernel spaces to invariant spaces. 
\begin{conjecture}\label{conj:main}\cite[Conjecture~3.2]{10.1007/978-3-031-75326-8_21}
For $f_1,\ldots,f_l,\, g \in \mathbb{K}\langle \underline{x}\rangle$, the following conditions are equivalent:
\begin{enumerate}
\item[(i)] for all $n \in \mathbb{N}$ and $\underline{X} \in \mathrm{M_n}(\mathbb{K})^d$, joint invariant subspaces for
$f_1(\underline{X}),\ldots,f_l(\underline{X})$ are invariant for $g(\underline{X})$;
\item[(ii)] $g$ belongs to the unital $\mathbb{K}$-algebra generated by $f_1,\ldots,f_l$.
\end{enumerate}
\end{conjecture}

 In~\cite{10.1007/978-3-031-75326-8_21}, the conjecture was verified for $l = 1$, and can be stated as the following proposition.
\begin{proposition}
    \cite[Proposition~3.3]{10.1007/978-3-031-75326-8_21}
Let $f, g \in \mathbb{K}\langle \underline{x} \rangle$. Then $g \in \mathbb{K}[f]$ if and only if for all $n \in \mathbb{N}$ and $\underline{X} \in M_n(\mathbb{K})^d$, the eigenvectors of $f(\underline{X})$ are the eigenvectors of $g(\underline{X})$.
\end{proposition}

Although the conjecture is true for  $l=1$, it remains open for  $l>1$, and this paper addresses this challenging case.

\subsection*{Our Contributions}


 Inspired by the proof of Theorem~\ref{Truezero}, we observe that the homogeneity of the generators can simplify the algebraic structure by restricting the degrees. Hence, we 
establish an equivalence between the geometric condition of invariant subspace inclusion and the algebraic condition of $g$ belonging to the subalgebra generated by homogeneous $f_1,\ldots,f_l$, i.e., Conjecture~\ref{conj:main} is true under the  assumption that $f_1,\dots,f_l$  are homogeneous.

\begin{theorem}[Homogeneous Nullstellensatz for Joint Invariant Subspaces] \label{thm:homo-case}
Let $f_1,\ldots,f_l, g \in \mathbb{K}\langle \underline{x}\rangle$, where $f_1,\ldots,f_l$ are \textbf{homogeneous}. Then the following are equivalent:
\begin{enumerate}
\item[(i)] for all $n \in \mathbb{N}$ and $\underline{X} \in \mathrm{M}_n(\mathbb{K})^d$, every joint invariant subspace of
$f_1(\underline{X}),\ldots,f_l(\underline{X})$ is invariant under $g(\underline{X})$;
\item[(ii)] $g$ belongs to the unital $\mathbb{K}$-algebra generated by $f_1,\ldots,f_l$.
\end{enumerate}
\end{theorem}

We next address Conjecture~\ref{conj:main} without assuming the homogeneity of $f_1,\ldots,f_l$. In particular, we show that Theorem~\ref{thm:homo-case} does not extend to the non-homogeneous generators. 

\begin{theorem}\label{thm:general-case}

   There exist non-homogeneous polynomials $f_1,\ldots,f_l, g \in \mathbb{K}\langle \underline{x}\rangle$ with $l \ge 2$ such that
for all $n \in \mathbb{N}$ and $\underline{X} \in \mathrm{M}_n(\mathbb{K})^d$, every joint invariant subspace of
$f_1(\underline{X}),\ldots,f_l(\underline{X})$ is invariant under $g(\underline{X})$,
while
$g$ does not belong to the unital $\mathbb{K}$-algebra generated by $f_1,\ldots,f_l$.
Consequently, Conjecture~\ref{conj:main} is not valid in general.
\end{theorem}

Rather than constructing an explicit counterexample to Conjecture~\ref{conj:main}, our conclusion follows from a computability gap: the subalgebra membership problem is undecidable in general (Theorem~\ref{thm:undecidable}), whereas the condition on invariant subspaces is decidable via the first-order theory of real closed fields (Theorem~\ref{thm:Tarski–Seidenberg}).  

By contrast, when the generators $f_1,\ldots,f_l$ are homogeneous, the corresponding subalgebra membership problem is known to be decidable~\cite{MR1805178}, which is consistent with the validity of the homogeneous Nullstellensatz for joint invariant subspaces established in this work.



\section{Homogeneous Nullstellensatz for Joint Invariant Subspaces }\label{sec2}

In this section, we prove that Conjecture~\ref{conj:main} holds under the assumption that the generators $f_1, \ldots, f_l$ are homogeneous. Our proof relies on constructions in the truncated Fock space.

Let $x_1,\ldots,x_d$ be $d$ noncommutative variables. 
The noncommutative truncated Fock space of order $m$ is a finite-dimensional $\mathbb{K}$-inner product space
\[
\mathcal{F}_m:= \mathrm{span}\big\{ \omega \mid \omega\in\langle\underline{x}\rangle,~\mathrm{deg}(\omega)\le m \big\},
\]
with $\big\{ \omega \mid \omega\in\langle\underline{x}\rangle,~ \mathrm{deg}(\omega)\le m \big\}$ being an orthonormal basis of~$\mathcal{F}_m$.

For each variable $x_k$, the left creation operator $L_{x_k}$ on Hilbert space $\mathcal{F}_m$ is defined by
\[
L_{x_k} ( \omega):=
\begin{cases}
x_k\omega, & \text{if } {\mathrm{deg}(\omega) < m},\\[2mm]
0, & \text{if } {\mathrm{deg}(\omega) = m}.
\end{cases}
\]

Let $\mathbb{K}\langle \underline{x}\rangle$ be the  $\mathbb{K}$-algebra of noncommutative polynomials. 
For a fixed positive integer $m$, let $I_m$ be the two-sided ideal   
\[
I_m := \Big\langle  \omega\mid \omega\in\langle\underline{x}\rangle,\; \operatorname{deg}(\omega) > m\Big\rangle,
\]
Then the algebra $\mathbb{K}\langle \underline{x}\rangle / I_m$ consists of elements which have the following form.
\[\hat f=\sum_{i=1}^{N} \alpha_{i} \omega_i + I_m ~,\; \text{where }\operatorname{deg}(\omega_i) \le m.\]

Let $V$ be a vector space over the field $\mathbb{K}$, and let $\underline{T} = (T_1,\ldots,T_l)$ be a tuple of linear operators on $V$. Let $\underline{y}=\{y_1,y_2,\ldots,y_l\}$ be a tuple of noncommutative variables, and $\langle \underline{y} \rangle$ be the set of words generated by $y_1,y_2,\ldots,y_l$. Let $u$ be an element in $\langle \underline{y} \rangle$. For a fixed vector $v \in V$ and $k \ge 0$, let
\begin{equation*}\label{vk}
    \mathcal{V}_k:= \mathrm{span} \bigl\{ u(\underline{T})v \mid u\in\langle\underline{y}\rangle,~ \text{deg} (u) \le k \bigr\}.
\end{equation*}
It follows that there is an ascending chain of subspaces
\begin{equation*}\label{ass}
\mathcal{V}_0 \subseteq \mathcal{V}_1 \subseteq \cdots \subseteq \mathcal{V}_k \subseteq \cdots
.
\end{equation*}
For such an ascending chain of subspaces, it follows that

\begin{lemma}\label{lem:invarint space}
For a finite-dimensional linear space $V$, the union 
\[
\mathcal{V}_\infty := \bigcup_{k \ge 0} \mathcal{V}_k
\]
is the smallest joint invariant subspace containing $v$ under the operators in $\underline{T}$. Furthermore, there exists an integer $a$ such that
\[
\mathcal{V}_a = \mathcal{V}_{a+1} = \mathcal{V}_{a+2} = \cdots=\mathcal{V}_\infty \subseteq V,
\]
and
\[
a \le \dim V.
\]
\end{lemma}

\begin{proof}

 
 The subspace  $\mathcal{V}_\infty$ is an invariant subspace since
  \[T_j \mathcal{V}_\infty =\bigcup_{k \ge 0} T_j\mathcal{V}_k\subseteq \bigcup_{k \ge 1} \mathcal{V}_k\subseteq \mathcal{V}_\infty,~\text{for all}~j=1\ldots l.\]
  If there exists an invariant subspace $W$ such that
  \[v\in W\subsetneq \mathcal{V}_\infty,\]
  then 
  \[u(\underline{T})v\in W,~\text{for all } u\in\langle\underline{y}\rangle. \]
  Therefore
  \[\mathcal{V}_\infty \subseteq W.\]
 Hence,  the subspace $\mathcal{V}_\infty$ is the smallest joint invariant subspace containing $v$ under the operators in $\underline{T}$.
  
    For $i\in\mathbb{N}$, if
  \[\mathcal{V}_i \subsetneq \mathcal{V}_{i+1}, \]
  then 
  \[\mathrm{dim}(\mathcal{V}_i)<\mathrm{dim}(\mathcal{V}_{i+1}).\]
  Since the space $V$ is a finite-dimensional linear space, there exists an integer $a$ such that
\[
\mathcal{V}_a = \mathcal{V}_{a+1} = \mathcal{V}_{a+2} = \cdots=\mathcal{V}_\infty \subseteq V,
\]
and
\[
a \le \dim V.
\]
\end{proof}
Now we are ready for our main result for  noncommutative homogeneous polynomials.

\begin{proof}[proof of Theorem~\ref{thm:homo-case}]
$\ $ 

(ii) $\implies$ (i) is trivial.

(i) $\implies$ (ii).
Let $\mathcal{F}_m$ be the truncated Fock space of order $m$ with $d$ noncommutative variables $x_1,\dots,x_d$, 
and let $\underline{L} = (L_{x_1},\ldots,L_{x_d})$ be the corresponding left creation operators. 

Consider the tuple of operators 
\[
F(\underline{L}) = (f_1(\underline{L}),\ldots,f_l(\underline{L})).
\]
Let vector  $v =\mathbf{1} \in \mathcal{F}_m$ (to avoid ambiguity, let $\mathbf{1}$ be the vector corresponding to the constant polynomial $1\in \mathbb{K}\langle \underline{x} \rangle$).

The ascending chain of subspaces generated by $v=\mathbf{1}$ and $F(\underline{L})$ is defined by:
\[\mathcal{V}_0 := \mathrm{span}\{v\},\; k=0,\] 
\[\mathcal{V}_k: = \mathrm{span}\{u(F(\underline{L})) v \mid u\in\langle\underline{y}\rangle,~\deg(u)  \le k\},\; k \ge 1.
\]

  Since $\mathcal{F}_m$ is finite-dimensional, Lemma~\ref{lem:invarint space} implies that
there exists a positive integer $a \le \dim \mathcal{F}_m$ 
such that $\mathcal{V}_a$ is the smallest joint invariant subspace of the operators $f_1(\underline{L}),\ldots,f_l(\underline{L})$ containing the vector $\mathbf{1}$. 
From the definition of the left creation operator \(L_{x_i}\) and of the space \(\mathcal{F}_m\), we have
\[u(F(\underline{L})) \mathbf{1} =0,~\text{when } u\in\langle\underline{y}\rangle,~\deg(u)  > m.\]
We conclude that \(a \le m\).


If condition (i) holds, then the joint invariant space of $f_1(\underline{L}),\ldots,f_l(\underline{L})$ is also invariant under $g(\underline{L})$. Hence, we have
\[g(L)\mathbf{1}\in\mathcal{V}_m. \]
Therefore, there exists a noncommutative polynomial $h_m \in \mathbb{K}\langle\underline{y}\rangle$ of degree at most $m$ such that 
\[h_m=\sum_{i=1}^{N_m} \alpha_{i}u_i ~\text{, where } u_i \in \langle \underline{y}\rangle ,~N_m\in \mathbb{N} \]
and
\[g=h_m(f_1,\ldots,f_l) \text{ mod } I_m.\]

Since $f_1,\ldots,f_l$ are noncommutative homogeneous polynomials, the noncommutative polynomial $u_i(f_1,\ldots,f_l)$ is also homogeneous.
For every $u_i$ that appears in $h_m$, define 
\[\beta_i:=
\begin{cases}
    \alpha_i, ~~\text{deg}(u_i(f_1,\ldots,f_l))\le m\\
    0,~~~~\text{deg}(u_i(f_1,\ldots,f_l))> m,
\end{cases}\]
\[h_m^{'}=\sum_{i=1}^{N_m} \beta_{i}u_i,\]
then $\operatorname{deg}(h_m^{'}) \leq m$.
\[g\equiv h_m^{'}(f_1,\ldots,f_l)\equiv h_m(f_1,\ldots,f_l) \text{ mod } I_m.\]
Since $m$ is arbitrary, let $m$ be an integer such that $m \ge \operatorname{deg}(g)$, then,
\[g=h_m^{'}(f_1,\ldots,f_l).\]
\end{proof}

\section{Non-homogeneous Case  } 

In this section, we demonstrate that the homogeneous Nullstellensatz for joint invariant subspace (Theorem~\ref{thm:homo-case}) cannot be generalized to the non-homogeneous generators. Consequently, Conjecture~\ref{conj:main} is not valid in general. Our proof is based on results on {undecidability} and {decidability} in computational  theory.

The first result shows the undecidability of the membership problem for finitely generated subalgebras of free associative algebras.

\begin{theorem}\cite[Corollary]{MR1286789}\label{thm:undecidable}
For any fixed field, the membership problem for finitely generated subalgebras of free associative algebras of rank at least $2$ is \textbf{undecidable}.
\end{theorem}

The second result is about the decidability of the first-order theory of real closed fields (cf.  \cite{MR63994} and \cite[Theorem~16]{MR1634189} ), where formulas in  first-order theory are constructed from polynomial equalities and inequalities about real variables, logical connectives (i.e., $\land$, $\lor$, $\lnot$ ) and quantifiers (i.e., $\exists$, $\forall$ ). 

\begin{theorem}[Tarski–Seidenberg]\label{thm:Tarski–Seidenberg}
   The first-order theory of real closed fields is \textbf{decidable}.
\end{theorem}
Now we are ready to discuss Theorem~\ref{thm:general-case} and Conjecture~\ref{conj:main}. Our proof is by contradiction. Assuming that the 
homogeneous Nullstellensatz for joint invariant subspaces (Theorem~\ref{thm:homo-case}) can be generalized to the case of non-homogeneous generators, then there exists an algorithm that solves the subalgebra membership problem in finite time, which contradicts Theorem~\ref{thm:undecidable}. We now complete the proof of Theorem~\ref{thm:general-case}, and thereby give a negative answer to Conjecture~\ref{conj:main}.

First, we present a procedure which is guaranteed to terminate if $g$ is in the subalgebra generated by $\{f_1, \ldots, f_l\}$. The term “procedure” is used rather than “algorithm” because it may not terminate on all inputs.

\begin{algorithm}[H]

\caption{Membership Verification}\label{alg:g-in-KF}
\SetKwProg{Procedure}{Procedure}{}{}
\Procedure{$\textsc{membership}(g,f_1,\ldots,f_l)$}{
    Initialize $n \leftarrow 0$\;
    \Repeat{termination condition met}{
        $n \leftarrow n + 1$\;
        Compute $\langle \underline{y} \rangle_n=\{u\in  \langle \underline{y}\rangle\mid \operatorname{deg}(u)\leq n \}$\;
      Check whether the linear equation $g=\sum_{ u \in \langle \underline{y} \rangle_n}c_{u}u(f)$ has a solution\;
        \If{there exists a solution  for the linear equation}{
            \Return{\textbf{Membership certified}}\;}
        }
    }

\end{algorithm}

Here, $\langle \underline{y} \rangle_n$ is the set of all  words generated by $\underline{y}=(y_1,\ldots,y_l)$ with degree at most $n$, and for any word $u \in \langle \underline{y}\rangle$,  $u(f)=u(f_1,\ldots,f_l)$ is obtained by replacing $y_1,\ldots,y_l$ in $u$ with $f_1,\ldots,f_l$.
Regarding Procedure~\ref{alg:g-in-KF}, we have the following lemma.

\begin{lemma}\label{lem:Memebr-term}

If $g$ belongs to the unital $\mathbb{K}$-algebra generated by $f_1,\ldots,f_l$, then Procedure~\ref{alg:g-in-KF} terminates in finite time on input  $(g,f_1,\ldots,f_l)$.
\end{lemma}
\begin{proof}
    It is clear that if $g$ belongs to the unital  $\mathbb{K}$-algebra generated by $f_1,\ldots,f_l$, then there exists a noncommutative polynomial $h \in \mathbb{K}\langle \underline{y}\rangle$ with 
    $ g=h(f_1,\ldots,f_l)$. Let $h$ be expressed as a finite linear combination of words, i.e.,
    \[h(\underline{y})=\sum_{i=1 }^N\ \alpha_{i}u_i(\underline{y}), \]
    with $c_{i} \in \mathbb{K}$ for all $i=1,2,\ldots,N$. Let $n'$ be the maximal value of 
    $\{\operatorname{deg}(u_i):i=1,2,\ldots,N\}$, then $ u_i\in \langle \underline{y} \rangle_{n'}$ for $i=1,2,\ldots,N$. By $ g=h(f_1,\ldots,f_l)$, we have
    \[g=\sum_{i=1 }^N \alpha_{i}u_i(f), \]
    which implies that the linear equation $g=\sum_{ u\in  \langle \underline{y} \rangle_{n'}}c_{u}u(f)$  at the $n'$-th loop of Procedure~\ref{alg:g-in-KF} admits a solution.
\end{proof}
The above conclusion implies the existence of a deterministic procedure that terminates  when $g$ belongs to the unital $\mathbb{K}$-algebra generated by $f_1,\ldots,f_l$. In conclusion, the set of instances of the subalgebra membership problem belongs to the complexity class \textsf{RE} (recursively enumerable).

It is clear that if there exist matrices $\underline{X} \in \mathrm{M_n}(\mathbb{K})^d$ and a joint invariant subspace  $V \subseteq \mathbb{K}^n$ for $f_1({\underline{X}}),\ldots,f_l({\underline{X}})$ such that $V$ is not an invariant subspace for $g(\underline{X})$, then $g$ does not belong to the unital $\mathbb{K}$-algebra generated by $f_1,\dots,f_l$. 
Based on this observation, a procedure is given to certify non-membership of the subalgebra over the field of complex numbers $\mathbb{C}$.

\begin{algorithm}[H]
\caption{Non-Membership Verification}\label{alg:g-notin-H}
\SetKwProg{Procedure}{Procedure}{}{}
\Procedure{$\textsc{Non-membership}(g,f_1,\ldots,f_l)$}{
    Initialize $n \leftarrow 0$\;
    \Repeat{termination condition met}{
        $n \leftarrow n + 1$\;
        \For{$r \leftarrow 1$ \KwTo $n$}{
            Check the existence of matrices $\underline{X} \in  \mathrm{M_{n}}(\mathbb{C})^d$ and $V \in  \mathrm{M_{n\times r}}(\mathbb{C})$ satisfying:\\
            \Indp
            1. $\operatorname{rank}(V)=r$\;
            2. The column space of $V$ is an invariant subspace for $f_j(\underline{X})$ for all $j=1,2,\ldots,l$\;
            3. The column space of $V$ is \textbf{not} an invariant subspace for $g(\underline{X})$\;
            \Indm
            \If{such matrices exist}{
                \Return{\textbf{Non-membership certified}}\;
            }
        }
    }
}
\end{algorithm}

The above procedure searches for \textit{counterexamples} across dimensions of the matrices and subspaces. Specifically, for each    pair of integers $n \geq 1$ and $1 \leq r \leq n$, it attempts to find matrices $X_1,\ldots,X_d \in  \mathrm{M_{n}}(\mathbb{C})$ and a subspace $V\subseteq \mathbb{C}^n$ (here we do not distinguish between the matrix and its column space) such that~$V$ is a joint invariant subspace for all $f_j(\underline{X})$, but is not an invariant subspace for $g(\underline{X})$. 

The following lemma shows that for fixed $n$ and $r$,   it is decidable  to determine whether there exist matrices $\underline{X}\in \mathrm{M_{n}}(\mathbb{C}) ^d$ and $V\in  \mathrm{M_{n\times r}}(\mathbb{C})$ that satisfy the conditions in Procedure~\ref{alg:g-notin-H}.
\begin{lemma}\label{lem:NonMemebr-term}
    For arbitrary fixed integers $n$ and $r$ with $1 \leq r\leq n$, it is \textbf{decidable} to determine the existence of the $n$-dimensional matrices  $\underline{X}=(X_1,\ldots,X_d)\in \mathrm{M_{n}}(\mathbb{C}) $ and an  $r$-dimensional subspace $V\subseteq  \mathbb{C}^n$ such that $V$ is invariant under $f_1(\underline{X}),\ldots,f_l(\underline{X})$, but not invariant under $g(\underline{X})$. 
\end{lemma}
\begin{proof}
Define the following formulas:
\begin{itemize}
    \item $\operatorname{IsFullRank}(V)\coloneqq \operatorname{det}(V^* V)>0$, where $V^*$ is the  Hermitian adjoint of $V$. This formula is used to determine whether $V$  is a matrix with full column rank.

    \item $\operatorname{IsInvSpace}(Y,V)\coloneqq \operatorname{minor}_{r+1}([YV|V])=0$. By checking the $(r+1)$-minors of the horizontally concatenated matrix $[YV|V] $, one can determine 
     if $V$ is an invariant subspace of $Y$. 
\end{itemize}
Then determining the existence of $X_1,\ldots,X_d$ and $V$ is equivalent to determining whether the following quantified formula holds.
\begin{equation}\label{equ:quantified-formula}
\begin{split}
    &\left(\exists X_1,X_2,...,X_d \in  \mathrm{M_{n}}(\mathbb{C}),  \exists V\in  \mathrm{M_{n\times r}}(\mathbb{C}) \right) \left(\operatorname{IsFullRank}(V) \right.\\ &
      \land   \bigwedge_{j=1}^l \operatorname{IsInvSpace}(f_j(\underline{X}),V)  
      \land \left( \lnot \operatorname{IsInvSpace}(g(\underline{X}),V)\right) ),
    \end{split}
\end{equation}
note that since each $f_j(\underline{X})$ is a polynomial of $\underline{X}$, every entry of $f_j(\underline{X})$ is also a polynomial in the entries of $\underline{X}$. For complex variables $X_k$ and $V$,  one can introduce corresponding real variables $ W_k $ and $ Z_k$ such that $X_k = W_k + \mathrm{i} Z_k $ for all indices $k=1,2,\ldots,d$, and  introduce $W_V, Z_V$   such that $V=W_V+\mathrm{i}Z_V$. Therefore, \eqref{equ:quantified-formula} can be rewritten as a first-order logical formula over the field of real numbers. Here, we use the notation $\underline{W}+\mathrm{i}\underline{Z}$ to  denote the $d$-tuple $(W_1 + \mathrm{i}Z_1, W_2 + \mathrm{i}Z_2, \ldots, W_d + \mathrm{i}Z_d)$.
\begin{equation*}
\begin{split}
    &  \left(\exists W_1,\ldots,W_d,Z_1,\ldots,Z_d  \in  \mathrm{M_{n}}(\mathbb{R}),  \exists W_V,Z_V\in  \mathrm\red{{M_{n\times r}}}(\mathbb{R})\right) \\
    &
   ( \operatorname{IsFullRank}( W_V+\mathrm{i}Z_V)
     \land   \bigwedge_{j=1}^l \operatorname{IsInvSpace}(f_j(\underline{W}+\mathrm{i}\underline{Z}),W_V+\mathrm{i}Z_V)  
      \land \\ & \left.
      \left( \lnot \operatorname{IsInvSpace}(g(\underline{W}+\mathrm{i}\underline{Z}),W_V+\mathrm{i}Z_V)\right) \right).
    \end{split}
\end{equation*}
By separating the real and imaginary parts of each polynomial, the above formula can be formally encoded into the first-order theory of the field of the real numbers. Consequently, according to Theorem~\ref{thm:Tarski–Seidenberg}, it is decidable.
\end{proof}

 Now we are ready for the proof of Theorem~\ref{thm:general-case}. The proof relies on the fact that a problem is decidable if and only if it is both recursively enumerable and co-recursively enumerable (see, e.g., \cite[Theorem~9.4]{10.5555/1196416}).

\begin{proof}[proof of Theorem~\ref{thm:general-case}]
First, note that if  the two conditions in  Conjecture~\ref{conj:main} are equivalent, then for $f_1,\ldots,f_l,\, g \in \mathbb{K}\langle {\underline{x}}\rangle$, the following two statements are also equivalent:
\begin{enumerate}
\item[(i)] There exist finite numbers $r$ and $n$ with  $ 1 \leq r \leq n \in \mathbb{N}$ and matrices $X_1,\ldots,X_d \in \mathrm{M_n}(\mathbb{K})$, $V \in \mathrm{M_{n\times r}}(\mathbb{K})$, such that the column space of $V$ is a joint invariant subspace for
$f_1({\underline{X}}),\ldots,f_l({\underline{X}})$,  but not an invariant space for $g({\underline{X}})$;
\item[(ii)] $g$ does not belong to the unital $\mathbb{K}$-algebra generated by $f_1,\ldots,f_l$.
\end{enumerate}

Therefore, if the conditions in  Conjecture~\ref{conj:main} are equivalent, then  for any noncommutative polynomials $f_1,\ldots,f_l,\, g \in \mathbb{C}\langle {\underline{x}}\rangle$ with $g$ not belonging to the unital $\mathbb{C}$-algebra generated by $f_1,\ldots,f_l$,  there exist finite-dimensional matrices $X_1,\ldots,X_d$ and $V$ as a certification of the non-membership, which implies that  Procedure~\ref{alg:g-notin-H} terminates in finite time on input $(g,f_1,\ldots,f_l)$  if $g$ doesn't belong to the unital $\mathbb{C}$-algebra generated by $f_1,\ldots,f_l$.

     Combining  with Lemma~\ref{lem:Memebr-term}, it follows that for any input $g$ and $f_1,\cdots,f_l$ in $\mathbb{C}\left<\underline{x}\right>$,  either Procedure~\ref{alg:g-in-KF} or Procedure~\ref{alg:g-notin-H} is guaranteed to terminate and give the correct membership decision.  Hence, by running both procedures in parallel, we obtain an algorithm that always terminates and correctly decides  the subalgebra membership problem, which implies that the membership problem for finitely generated subalgebras of free associative algebras over $\mathbb{C}$  is decidable, which contradicts Theorem~\ref{thm:undecidable}.
\end{proof} 

\section{Further Discussion}
As demonstrated in this paper, the  Conjecture \ref{conj:main} does not hold for finite-dimensional matrices in the general case (non-homogeneous,$l\geq 2$). A natural generalization is to investigate whether the equivalence holds when evaluating polynomials on infinite-dimensional operators. We therefore propose the following conjecture.

\begin{conjecture}\label{conj:infinite-dim}
Let $f_1,\ldots,f_l,\, g \in \mathbb{C}\langle \underline{x}\rangle$. The following are equivalent:
\begin{enumerate}
\item[(i)] for every complex Hilbert space $\mathcal{H}$ and every tuple of bounded operators $\underline{A}=(A_1,\ldots,A_d) \in \mathcal{B}(\mathcal{H})^d$, every joint invariant subspace of $f_1(\underline{A}),\ldots,f_l(\underline{A})$ is invariant under $g(\underline{A})$;
\item[(ii)] $g$ belongs to the unital $\mathbb{C}$-algebra generated by $f_1,\ldots,f_l$.
\end{enumerate}
\end{conjecture}

We provide a preliminary discussion of this conjecture based on the Fock space. Consider the full Fock space $\mathcal{F}$, defined as the Hilbert space having the free monoid $\langle \underline{x} \rangle$ (the set of all words in variables $x_1, \ldots, x_d$) as an orthonormal basis, that is,

 \[\mathcal{F}:=\overline{\mathrm{span}\{~\omega \mid \omega\in \langle\underline{x}\rangle\}}.\]
 For each variable $x_k$, the left creation operators are defined by 
 \[L_{x_k} (\omega) :=x_k\omega.\]

By applying the arguments in Section \ref{sec2} to the general case (without the homogeneity assumption), condition (i) of Conjecture~\ref{conj:infinite-dim} implies the following conclusion:
 
 \[g(\underline{L})\mathbf{1}\in \overline{\mathrm{span}\{~u(F(\underline{L})\mathbf{1}\mid u\in \langle\underline{y}\rangle\}}.\]
  There still exists a gap between the above result and condition (ii) in Conjecture~\ref{conj:infinite-dim}, which is equivalent to the following condition:
 \[g(\underline{L})\mathbf{1}\in {\mathrm{span}}\{u(F(\underline{L}))\mathbf{1}\mid u\in \langle\underline{y}\rangle\}.\]

 Our disproof of the finite-dimensional case exploited the discrepancy between undecidability of the subalgebras membership problem 
and decidability of the first-order theory of a real
closed field.  In the infinite-dimensional setting of Conjecture~\ref{conj:infinite-dim}, this discrepancy vanishes. While the first-order theory of real closed fields is decidable, this property does not extend to operator algebras (see Fritz~\cite{MR4240748}). Consequently, the decidability obstruction is removed and resolving the conjecture in this setting requires a new approach.
 

\bibliographystyle{plain}
\bibliography{sample}
\end{document}